\documentclass[10pt, a4paper, reqno, twoside]{amsart}

\usepackage{amsmath, amsopn, amsfonts, amsthm, amssymb, amscd, enumerate, multicol, scalefnt}
\usepackage{sansmath}
\usepackage{etex}
\usepackage{hyperref}
\usepackage{t1enc}
\usepackage[latin1]{inputenc}
\usepackage{graphicx}
\usepackage[all]{xy}
\usepackage{color}
\usepackage{slashed}
\usepackage[mathscr]{euscript}
\usepackage{enumitem}
\setlist[itemize]{noitemsep, topsep=1pt, leftmargin=30pt}
\newcommand\bcdot{\ensuremath{
  \mathchoice
   {\mskip\thinmuskip\lower0.2ex\hbox{\scalebox{1.6}{$\cdot$}}\mskip\thinmuskip}}
   {\mskip\thinmuskip\lower0.2ex\hbox{\scalebox{1.6}{$\cdot$}}\mskip\thinmuskip}
   {\lower0.3ex\hbox{\scalebox{1.2}{$\cdot$}}}
   {\lower0.3ex\hbox{\scalebox{1.2}{$\cdot$}}}
}
\theoremstyle{plain}
\newtheorem{theo}{Theorem}[section]

\newtheorem{prop}[theo]{Proposition}

\theoremstyle{definition}
\newtheorem{rem}[theo]{Remark}

\theoremstyle{plain}
\newtheorem{thmint}{Theorem}

\theoremstyle{plain}
\newtheorem{lemma}[theo]{Lemma}
\newtheorem{theorem}[theo]{Theorem}
\newtheorem{corollary}[theo]{Corollary}
\newtheorem{proposition}[theo]{Proposition}

\theoremstyle{definition}

\renewcommand{\=}{:=}

\renewcommand{\a}{\alpha}

\renewcommand{\d}{\delta}
\newcommand{\e}{\varepsilon}
\newcommand{\f}{\varphi}
\newcommand{\g}{\gamma}

\renewcommand{\l}{\lambda}
\newcommand{\w}{\omega}
\newcommand{\q}{\vartheta}

\newcommand{\Q}{\Theta}

\newcommand{\W}{\Omega}

\newcommand{\bR}{\mathbb{R}}
\newcommand{\bZ}{\mathbb{Z}}

\newcommand{\bN}{\mathbb{N}}

\newcommand{\fG}{\mathsf{G}}
\newcommand{\fH}{\mathsf{H}}

\renewcommand{\gg}{\mathfrak{g}}
\newcommand{\gh}{\mathfrak{h}}

\newcommand{\cA}{\mathcal{A}}

\newcommand{\cC}{\mathcal{C}}

\newcommand{\cI}{\mathcal{I}}

\newcommand{\eA}{\EuScript{A}}
\newcommand{\eB}{\EuScript{B}}

\newcommand{\eD}{\EuScript{D}}

\newcommand{\eJ}{\EuScript{J}}

\newcommand{\eU}{\EuScript{U}}
\newcommand{\eV}{\EuScript{V}}
\newcommand{\eW}{\EuScript{W}}

\newcommand{\p}{\partial}

\newcommand{\rar}{\rightarrow}

\renewcommand{\square}{\kern1pt\vbox
{\hrule height 0.6pt\hbox{\vrule width 0.6pt\hskip 3pt \vbox{\vskip
6pt}\hskip 3pt\vrule width 0.6pt}\hrule height0.6pt}\kern1pt}
\renewcommand{\=}{\  \raisebox{0.15mm}{:} {=} \ }

\DeclareMathOperator\Lie{Lie}

\DeclareMathOperator\diff{d}

\DeclareMathOperator\Ad{Ad}

\DeclareMathOperator\Id{Id}

\DeclareMathOperator\Diff{Diff}

\newcommand{\td}{\mathtt{d}}

\newcommand{\wt}{\widetilde}

\newcommand{\zero}{\operatorname{o}}
\def\<#1,#2>{\langle\,#1,\,#2\,\rangle}

\newcommand{\aac}{\`a}
\newcommand{\Aac}{\`A}

\newcommand{\Math}{{\it Mathematica\raise5 pt\hbox{$\scriptscriptstyle \circledR$}7}}
\newcommand{\beq}{\begin{equation}}
\newcommand{\eeq}{\end{equation}}
\def\<#1,#2>{\langle\,#1,\,#2\,\rangle}
\newcommand{\arr}{\begin{array}{rlll}}
\newcommand{\ea}{\end{array}}
\newcommand{\bea}{\begin{eqnarray}}
\newcommand{\eea}{\end{eqnarray}}
\newcommand{\bean}{\begin{eqnarray*}}
\newcommand{\eean}{\end{eqnarray*}}

\def\sideremark#1{\ifvmode\leavevmode\fi\vadjust{
\vbox to0pt{\hbox to 0pt{\hskip\hsize\hskip1em
\vbox{\hsize3cm\tiny\raggedright\pretolerance10000
\noindent #1\hfill}\hss}\vbox to8pt{\vfil}\vss}}}

\newcounter{ssig}
\newcounter{ttig}
\newcommand{\Iso}{\operatorname{Iso}}

\title[A local version of the Myers-Steenrod Theorem]{A local version of the Myers-Steenrod Theorem}
\author{Francesco Pediconi}

\thanks{This work was supported by GNSAGA of INdAM}

\begin{document}
\begin{abstract}
We prove the Myers-Steenrod Theorem for local topological groups of isometries acting on pointed $\cC^{k,\a}$-Riemannian manifolds, with $k+\a>0$. As an application, we infer a new regularity result for a certain class of locally homogeneous Riemannian metrics. \end{abstract}

\maketitle


\section{Introduction} \setcounter{equation} 0

In this paper we give a characterization of local groups of isometries that admit structures of local Lie transformation groups. More precisely, we prove the following

\begin{thmint} Any locally compact and effective local topological group of isometries acting on a pointed $\cC^{k,\a}$-Riemannian manifold, with $k+\a>0$, is a local Lie group of isometries. \label{MAIN1} \end{thmint}

Our result can be considered as a local version of the Myers-Steenrod Theorem \cite{MS}. We recall that the most enhanced version of this result is actually a consequence of the celebrated Gleason, Montgomery and Zippin solution to the Hilbert fifth problem \cite{Gle, MZ1}: \begin{itemize}[leftmargin=37pt]
\item[(H5)] {\it A locally compact topological group admits a Lie group structure if and only if it is locally Euclidean, and this occurs if and only if it has no small subgroups.}
\end{itemize} Note that (H5) is a characterization of the Lie groups among all topological groups in terms of just group theory and topology. It was thus natural to expect that a similar property holds for local Lie groups too. However, such a result was proved only recently by Goldbring in \cite{Go} using techniques from non-standard Analysis. The proof of our Theorem \ref{MAIN1} is strongly based on Goldbring's Theorem. \smallskip

For a better understanding of our result, it is convenient to briefly review the relations between the solution to the Hilbert fifth problem (H5) and the various known versions of the Myers-Steenrod Theorem. We start recalling that the original paper \cite{MS} contains the following two results: \begin{itemize}[leftmargin=37pt]
\item[(MS1)] {\it Any distance preserving map between $\cC^k$-Riemannian manifolds, with $k\geq 2$, is of class $\cC^{k-1}$;}
\item[(MS2)] {\it Any closed group of isometries acting on a $\cC^k$-Riemannian manifold, with $k \geq 2$, is a Lie transformation group.}
\end{itemize} Subsequently, the works by Calabi and Hartman \cite{CH}, Re$\check{\rm s}$etnjak \cite{Re}, Sabitov \cite{Sab} and Shefel' \cite{Sh} allowed to obtain the following stronger version of (MS1): \begin{itemize}[leftmargin=40pt]
\item[(MS1')] {\it Any distance preserving map between $\cC^{k,\a}$-Riemannian manifolds, with $k+\a>0$, is of class $\cC^{k+1,\a}$.}
\end{itemize} Now, claims (MS1') and (H5) imply a strengthened version of (MS2), which holds under much lower regularity assumptions. Namely \begin{itemize}[leftmargin=40pt]
\item[(MS2')] {\it Any closed group of isometries acting on a $\cC^{k,\a}$-Riemannian manifold, with $k+\a >0$, is a Lie transformation group.}
\end{itemize} To the best of our knowledge, this is the strongest version of the Myers-Steenrod Theorem which can be obtained using the so far known results. For the reader convenience, we provide a proof of it in Section \ref{dimMS2} below. Our Theorem \ref{MAIN1} is obtained under the same regularity assumptions of such stronger version of (MS2) and can therefore be considered as a perfect analogue of it in the category of the local groups of transformations.

We would like to point out that, as many authors have predicted the existence of a local version of (H5), also the contents of our Theorem \ref{MAIN1} were expected to be true (see \cite[p. 616]{Mos}). On the other hand, its proof remained an open problem since the very first appearance of the Myers-Steenrod Theorem in \cite{MS}, where the authors themselves ended the paper asking explicitly whether any locally compact group germ of local isometries were a Lie group germ or not. We guess that the lapse of time that passed between the statement of the problem and the finding of the solution presented in this paper was caused by the lack of specific technical tools for dealing with local Lie groups, a gap which was finally filled in the previous quoted paper by Goldbring. \smallskip

As a by-product, we also obtain a useful regularity result for locally homogeneous Riemannian metrics, on which the main results of \cite{Ped2,Ped3} are based. We recall that a Riemannian manifold $(M,g)$ is called {\it locally homogeneous} if the pseudogroup of local isometries of $(M,g)$ acts transitively on $M$. It is known that a smooth locally homogeneous Riemannian manifold is necessarily real analytic (see e.g. \cite[Lemma 2.1]{BLS}). From this and \cite{No}, it follows that for any $p \in M$ there exists a local Lie group of isometries which acts transitively on $(M,g,p)$. Here, we point out that transitivity for local group actions means that the orbit of the local group through the distinguished point $p$ contains a whole neighborhood of it.

By means of Theorem \ref{MAIN1}, the following kind of converse holds true.

\begin{thmint} Let $(M,g)$ be a locally homogeneous $\cC^1$-Riemannian manifold. If there exist a point $p\in M$ and a locally compact, effective local topological group of isometries which acts transitively on $(M,g,p)$, then $(M,g)$ is real analytic. \label{MAIN2} \end{thmint} \smallskip

The paper is structured as follows. In Sections \ref{prelRG} and \ref{prelGT} we collect some preliminaries on Riemannian manifolds of low regularity and groups of transformations, respectively. In Section \ref{dimMS2} we give the advertised proof of claim (MS2'), from which we derive a regularity result for homogeneous Riemannian metrics. In Sections \ref{thmA} and \ref{thmB} we prove Theorem \ref{MAIN1} and Theorem \ref{MAIN2}, respectively. \medskip

\noindent{\it Acknowledgement.} We warmly thank Christoph B\"ohm and Luigi Verdiani for helpful discussions about several aspects of this paper. We are grateful to Andrea Spiro for his interest and his important suggestions. We also thank Isaac Goldbring, Linus Kramer, Richard Palais and Fabio Podest{\aac} for useful comments. Finally, we would like to thank the anonymous referee for his/her careful reading of the manuscript. \medskip

\section{Riemannian manifolds of low regularity} \label{prelRG} \setcounter{equation} 0

\subsection{Notation} \hfill \par

%

We use the following standard notation (see \cite[p. 52]{GT}). Given a pair $(k,\a) \in \big(\bZ_{\geq0} \times (0,1]\big) \cup \{(\infty,0)\}$ and a ball $B \subset\subset \bR^m$, we denote by $\cC^{k,\a}(\bar{B})$ the subspace of functions in $\cC^{k}(\bar{B})$ whose $k$-th order partial derivatives are uniformly $\a$-H\"older continuous in $B$. By convention, we set $\cC^{k,0}(\bar{B})\=\cC^{k}(\bar{B})$ so that the notation $\cC^{k,\a}(\bar{B})$ is meaningful for all $\a \in [0,1]$. We recall that, for $(k,\a)\neq(\infty,0)$, the space $\cC^{k,\a}(\bar{B})$ turns into a Banach space when it is endowed with the norm $$|\!|f|\!|_{\cC^{k,\a}(\bar{B})} \= \begin{cases}
|\!|f|\!|_{\cC^k(\bar{B})} & \text{ if } \a=0 \\
|\!|f|\!|_{\cC^k(\bar{B})}+\max_{|j|=k}|\!|\p^{j}f|\!|_{\a,B} & \text{ if } \a\neq0 \end{cases} \,\, ,$$ where $$|\!|f|\!|_{\cC^k(\bar{B})} \= \sum_{s=0}^k\max_{|j|=s}\left\{\sup_{x\in B}\big|(\p^{j}f)(x)\big|\right\} \,\, , \quad |\!|u|\!|_{\a,B} \= \sup_{x,y \in B}\frac{\big|u(x)-u(y)\big|}{|x-y|^{\a}} \,\, .$$ Here, $j=(j_1,{\dots},j_m)$ is a multi-index, $|j|\=j_1+{\dots}+j_m$ and $\p^{j}f\=\frac{\p^{|j|}f}{\p^{j_1}{x^1}{\dots}\p^{j_m}{x^m}}$.

We say that a function $F=(F^1,{\dots},F^q): U \subset \bR^m \rar \bR^q$ is {\it of class $\cC^{k,\a}$} if $F^i|_{B} \in \cC^{k,\a}(\bar{B})$ for any $1\leq i \leq q$ and for any ball $B \subset\subset U$. In what follows, {\it smooth} will always mean $\cC^{\infty}$-smooth.


A path $\g: I \subset \bR \rar \bR^m$ is said to be {\it of class $\cA\cC$} if for any closed subinterval $[a,b] \subset I$, the restriction $\g|_{[a,b]}$ is absolutely continuous. We recall that if $\g:I \rar \bR^m$ is a path of class $\cA\cC$, then the tangent vector $\dot{\g}(t) \in T_{\g(t)}\bR^m=\bR^m$ exists for almost all $t \in I$.

\subsection{Riemannian metrics of low regularity} \hfill \par

Let $M$ be a topological manifold. From now on, every manifold is assumed to be connected. An atlas $\eA$ on $M$ is said to be a {\it $\cC^{k,\a}$-atlas} if its overlap maps are of class $\cC^{k,\a}$. A $\cC^{k_1,\a_1}$-atlas $\eA_1$ and a $\cC^{k_2,\a_2}$-atlas $\eA_2$ on $M$, with $k_1<k_2$ or $(k_1=k_2) \land (\a_1\leq\a_2)$, are said to be {\it compatible} if their union $\eA_1\cup\eA_2$ is a $\cC^{k_1,\a_1}$-atlas on $M$. The following classical result guarantees the existence of smooth structures under far weaker hypotheses.

\begin{theorem}[\cite{Hi}, Thm 2.9] Let $M$ be a topological manifold and $\eA$ be a $\cC^k$-atlas on $M$. If $k\geq1$, then there exists a smooth atlas $\eA_1$ on $M$ compatible with $\eA$. Moreover, if $\eA_2$ is another smooth atlas on $M$ compatible with $\eA$, then $(M,\eA_1)$ is smoothly diffeomorphic to $(M,\eA_2)$. \end{theorem}

This theorem allows us to restrict our discussion, from now on, to the realm of smooth manifolds. On this regard, we recall the following standard definitions: \begin{itemize}
\item[$\bcdot$] a function $f:M_1 \rar M_2$ between smooth manifolds is said to be {\it of class $\cC^{k,\a}$} if its expressions in local coordinates are of class $\cC^{k,\a}$;
\item[$\bcdot$] a tensor field $T$ is said to be {\it of class $\cC^{k,\a}$} if its components in local coordinates are of class $\cC^{k,\a}$;
\item[$\bcdot$] a path $\g: I \subset \bR \rar M$ on a smooth manifold is said to be {\it of class $\cA\cC$} if its expressions in local coordinates are of class $\cA\cC$.
\end{itemize} A {\it $\cC^{k,\a}$-Riemannian manifold $(M,g)$} is the datum of a smooth manifold $M$ together with a Riemannian metric $g$ on $M$ of class $\cC^{k,\a}$, that is $g(\tfrac{\p}{\p x^i},\tfrac{\p}{\p x^j})$ are of class $\cC^{k,\a}$ for any choice of coordinate vector fields $\tfrac{\p}{\p x^1},{\dots},\tfrac{\p}{\p x^m}$. We define \beq \begin{gathered} \cI\=\big\{\g:[a_{\g},b_{\g}] \rar M \text{ path of class $\cA\cC$ } \big\} \,\, , \\
\ell_g: \cI \rar \bR \,\, , \quad \ell_g(\g) \= \int_{a_{\g}}^{b_{\g}}\big|\dot{\g}(t)\big|_g\,dt \,\, , \\
\td_g: M\times M \rar \bR \,\, , \quad \td_g(x,y) \= \inf\big\{\ell_g(\g): \g \in \cI , \, \g(a_{\g})=x , \, \g(b_{\g})=y \big\} \,\, .
\end{gathered} \label{length} \eeq

\begin{proposition}[\cite{Bu}] Let $(M,g)$ be a $\cC^0$-Riemannian manifold and $(\cI, \ell_g, \td_g)$ as in \eqref{length}. \begin{itemize}
\item[i)] The map $\ell_g$ is additive with respect to concatenation, continuous on segments and invariant under reparametrizations.
\item[ii)] The map $\td_g$ is a distance function and it determines the same topology of $M$.
\item[iii)] Given a path $\g \in \cI$, the following equalities hold: $$\begin{gathered}
|\dot{\g}(t)|_g=\lim_{\d\rar0}\frac{\td_g(\g(t+\d),\g(t))}{\d} \,\,\, \text{ for any $t \in (a_{\g},b_{\g})$ for which $\dot{\g}(t)$ exists } \,\, , \\
\ell_g(\g){=}\sup\bigg\{\sum_{i=1}^N\td_g(\g(t_{i-1}),\g(t_i)) : N \in \bN \, , \,\, a_{\g}{=}t_0 {<} t_1 {<} {\dots} {<} t_N {=}b_{\g}\bigg\} \,\, . \end{gathered}$$
\end{itemize} \end{proposition}

This last result shows that the triple $(\cI,\ell_g,\td_g)$ defined in \eqref{length} turns a $\cC^{k,\a}$-Riemannian manifold $(M,g)$ into a separable, locally compact length space. From now on, we will use the notation $\eB_g(x,r)$ to denote the metric ball centered at $x \in M$ of radius $r>0$ inside $(M,\td_g)$. \smallskip

Given two $\cC^{k,\a}$-Riemannian manifolds $(M_1,g_1)$ and $(M_2,g_2)$, a function $f: M_1 \rar M_2$ is said to be a {\it metric isometry} if it is surjective and distance preserving, i.e. $d_{g_1}(x,y)=d_{g_2}(f(x),f(y))$ for any $x,y \in M_1$. It is straightforward to observe that any metric isometry is a $\cC^{0,1}$-homeomorphism and that the inverse of a metric isometry is itself a metric isometry. On the other hand, a map $f: M_1 \rar M_2$ is called a {\it Riemannian isometry} if it is a $\cC^{k+1,\a}$-diffeomorphisms between $M_1$ and $M_2$ such that $f^*g_2=g_1$. Notice that any Riemannian isometry is, in particular, a metric isometry. Remarkably, the following weaker converse assertion holds.

\begin{theorem}[\cite{CH, Re, Sab, Sh},\cite{Tay}] Let $f:(M_1,g_1) \rar (M_2,g_2)$ be a metric isometry between two $\cC^{k,\a}$-Riemannian manifolds. If $k+\a>0$, then $f$ is of class $\cC^{k+1,\a}$ and it is a Riemannian isometry. \label{Isoreg} \end{theorem}

From now on, we will use the term {\it isometry} just to indicate a metric isometry. By means of Theorem \ref{Isoreg}, this coincides with the notion of Riemannian isometry only one exception, namely the pathological case $k+\a=0$. The full isometry group of a $\cC^{k,\a}$-Riemannian manifold $(M,g)$ will be denoted by $\Iso(M,g)$. \medskip

\section{Groups of transformations} \label{prelGT} \setcounter{equation} 0

\subsection{Global groups of transformations} \hfill \par

We recall that a {\it topological group} is a Hausdorff topological space equipped with a continuous group structure. A topological group is a {\it (real analytic) Lie group} if it is endowed with a smooth (resp. real analytic) manifold structure with respect to which the group operations are smooth (resp. real analytic). It is well known that the category of real analytic Lie groups is equivalent to the category of smooth Lie groups via the forgetful functor (see e.g. \cite[p. 43]{KN1}). The following characterization of Lie groups is the above mentioned solution to the Hilbert fifth problem.

\begin{theorem}[\cite{Gle, MZ1}] For any connected and locally compact topological group $\fG$, the following conditions are equivalent. \begin{itemize}
\item[a)] $\fG$ is locally Euclidean, i.e. there is a neighborhood of $e\!\in\!\fG$ homeomorphic to an open ball of some $\bR^N\!$.
\item[b)] $\fG$ has no small subgroups (NSS for short), i.e. there exists a neighborhood of $e \in \fG$ containing no nontrivial subgroups of $\fG$.
\item[c)] $\fG$ admits a unique smooth manifold structure which makes it a Lie group.
\end{itemize} \label{H5} \end{theorem}

Let $M$ be a smooth manifold. Given $k \in \bZ \cup \{\infty\}$, $k\geq0$, a {\it topological group of $\cC^k$-transformations $G =(\fG,\Q)$ on $M$} is the datum of a topological group $\fG$ together with a continuous action $\Q: \fG\times M \rar M$ on $M$ such that the map $\Q(a)\= \Q(a,\cdot) : M \rar M$ is of class $\cC^k$ for any $a \in \fG$. We recall that the correspondence $a \mapsto \Q(a)$ determines a group homeomorphism $\fG \rar \Diff^k(M)$ from $\fG$ to the group of $\cC^k$-diffeomorphisms of $M$ and that $G=(\fG,\Q)$ is called {\it effective} (resp. {\it almost-effective}) if the kernel of $\fG \rar \Diff^k(M)$ is trivial (resp. discrete). Furthermore, we say that $G=(\fG,\Q)$ is {\it closed} if it is effective and $\Q(\fG)$ is closed in $\Diff^k(M)$. \smallskip

A topological group of $\cC^k$-transformations $G=(\fG,\Q)$ on $M$ is called {\it Lie group of $\cC^k$-transformations} if $\fG$ is a Lie group and the map $\Q:\fG\times M \rar M$ is of class $\cC^k$.

\begin{rem} By \cite[Thm 4]{BM}, the second condition above is redundant. Namely, if $\fG$ is a Lie group and each map $\Q(a):M\rar M$ is of class $\cC^k$, then the map $\Q:\fG\times M \rar M$ is automatically of class $\cC^k$. \label{BMT4} \end{rem}

If $M$ is equipped with a $\cC^{k,\a}$-Riemannian metric $g$, then $G=(\fG,\Q)$ is called {\it topological (resp. Lie) group of isometries} if each map $\Q(a): M \rar M$ is an isometry of $(M,g)$. By the classical Myers-Steenrod Theorem, it is known that {\it any closed topological group of isometries of a $\cC^k$-Riemannian manifold, with $k\geq2$, is a Lie group of isometries.}

\subsection{Local groups of transformations} \hfill \par

In this subsection, we collect some basic facts on local groups of transformations. For more details, we refer to \cite[Ch 1]{Ol}, \cite{Pa}, \cite{PoSp}.

\subsubsection{Local topological groups and local Lie groups} \hfill \par

A {\it local topological group} is a tuple $(\fG, e, \eJ(\fG), \eD(\fG), \jmath, \nu)$ formed by: \begin{itemize}
\item[i)] a Hausdorff topological space $\fG$ with a distinguished element $e \in \fG$ called {\it unit},
\item[ii)] a neighborhood $\eJ(\fG) \subset \fG$ of $e$ and an open subset $\eD(\fG) \subset \fG\times\fG$ which contains both $\fG\times\{e\},\{e\}\times \fG$,
\item[iii)] two continuous maps $\jmath: \eJ(\fG) \rar \fG$, $\nu: \eD(\fG) \rar \fG$,
\end{itemize} such that, for any choice of $a,a_1,a_2 \in \fG$ and $b \in \eJ(\fG)$: \begin{itemize}
\item[$\bcdot$] $\nu(a,e)=\nu(e,a)=a$,
\item[$\bcdot$] if $(a_1,a), (a,a_2), (a_1,\nu(a,a_2)), (\nu(a_1,a),a_2) \in \eD(\fG)$, then it holds that $\nu(a_1,\nu(a,a_2))= \nu(\nu(a_1,a),a_2)$,
\item[$\bcdot$] $(b,\jmath(b)),(\jmath(b),b) \in \eD(\fG)$ and $\nu(b,\jmath(b))=\nu(\jmath(b),b)=e$. \end{itemize} From now on we adopt the usual notation $a_1\cdot a_2 \= \nu(a_1,a_2)$, $a^{-1}\=\jmath(a)$ and we will indicate any local topological group $(\fG, e, \eJ(\fG), \eD(\fG), \jmath, \nu)$ simply by $\fG$.

Given a local topological group $\fG$, every neighborhood $\eU$ of the unit $e \in \fG$ inherits a structure of local topological group induced by $\fG$. In fact, if we set $$\eJ(\eU) \= \eJ(\fG) \cap \eU \cap \jmath^{-1}(\eU) \,\, , \quad \eD(\eU)\=\eD(\fG) \cap (\eU \times \eU) \cap \nu^{-1}(\eU) \,\, ,$$ then one can directly check that $(\eU, e, \eJ(\eU), \eD(\eU), \jmath|_{\eJ(\eU)}, \nu|_{\eD(\eU)})$ is itself a local topological group. In this case, we say that {\it $\eU$ is a restriction of $\fG$}. We remark that $\fG$ can be restricted to a neighborhood $\eU$ of the unit which is {\it symmetric}, i.e. $\eU=\eJ(\eU)$, and {\it cancellative}, i.e. for any $a,a_1,a_2 \in \eU$ it holds: \begin{itemize}
\item[$\bcdot$] if $(a,a_1), (a,a_2) \in \eD(\eU)$ and $a\cdot a_1=a\cdot a_2$, then $a_1=a_2$;
\item[$\bcdot$] if $(a_1,a), (a_2,a) \in \eD(\eU)$ and $a_1\cdot a=a_2\cdot a$, then $a_1=a_2$;
\item[$\bcdot$] if $(a_1,a_2) \in \eD(\eU)$, then $(a_2^{-1},a_1^{-1})\in \eD(\eU)$ and $(a_1\cdot a_2)^{-1}=a_2^{-1}\cdot a_1^{-1}$
\end{itemize} (see e.g. \cite[Sec 1.5.6]{Tao}, \cite[Cor 2.17]{Go}). In particular, this implies that $(\jmath|_{\eU} \circ \jmath|_{\eU}) = \Id_{\eU}$. From now on, any local topological group $\fG$ and any neighborhood $\eU \subset \fG$ of the unit are assumed to be symmetric and cancellative. \smallskip

A subset $\fH \subset \fG$ which contains the unit $e \in \fG$ is said to be a {\it sub-local group}, if there exists a neighborhood $\eV$ of $\fH$ such that for any $a \in \fH$, $(a_1,a_2) \in (\fH \times \fH) \cap \eD(\fG)$ it holds $$ a^{-1} \in \eV \,\,\Longrightarrow\,\, a^{-1} \in \fH \,\, , \quad\quad a_1\cdot a_2 \in \eV \,\,\Longrightarrow\,\, a_1\cdot a_2 \in \fH \,\, .$$ Any such an open subset $\eV \subset \fG$ is called {\it associated neighborhood for $\fH$}.

A sub-local group $\fH$ such that $\fH \times \fH \subset \eD(\fG)$ and with $\eV=\fG$ is called a {\it subgroup}. Notice that, by such hypothesis, $a\cdot b \in \fH$, $a^{-1} \in \fH$ for any $a,b \in \fH$ and therefore $\fH$ is a topological group in the usual sense. The local topological group $\fG$ is said {\it to have no small subgroups} (NSS) if there exists a neighborhood of the unit with no nontrivial subgroups. \smallskip

Given two local topological groups $\fG$ and $\fG'$, a {\it local homomorphism from $\fG$ to $\fG'$} is a pair $(\eU,\f)$ given by a neighborhood $\eU \subset \fG$ of the unit and a continuous function $\f: \eU \rar \fG'$ such that \begin{itemize}
\item[$\bcdot$] $\f(e)=e'$ and $\f(\eD(\eU))\subset\eD(\fG')$,
\item[$\bcdot$] $\f(a^{-1})=\f(a)^{-1}$ and $\f(a_1\cdot a_2)=\f(a_1)\cdot \f(a_2)$ for any $a \in \eU$, $(a_1,a_2) \in \eD(\eU)$.
\end{itemize} Two local homomorphisms $(\eU_1,\f_1)$, $(\eU_2,\f_2)$ are {\it equivalent} if there exists a neighborhood $\tilde{\eU} \subset \eU_1 \cap \eU_2$ of the unit such that $\f_1|_{\tilde{\eU}}=\f_2|_{\tilde{\eU}}$. For the sake of shortness, we will simply write $\f: \fG \rar \fG'$ to denote a local homomorphism $(\eU,\f)$, determined up to an equivalence. The composition $\f' \circ \f$ of two local homomorphisms is defined in an obvious way and a local homomorphism $\f: \fG \rar \fG'$ is called a {\it local isomorphism} if there exists a local homomorphism $\psi:\fG' \rar \fG$ such that $\psi \circ \f=\Id_{\fG}$ and $\f \circ \psi=\Id_{\fG'}$, where of course the equalities are up to equivalence. \smallskip

A {\it local Lie group} is a local topological group that is also a smooth manifold in such a way that the local group operations $\jmath:\fG\rar \fG$ and $\nu:\eD(\fG) \rar \fG$ are smooth. Just like in the global Lie groups theory, one can associate a Lie algebra $\gg$ of left invariant vector fields to any local Lie group $\fG$. Analogues of Lie's three fundamental theorems hold also for the local Lie groups (\cite[Ch 3]{Bou}). In particular, it turns out that every local Lie group is locally isomorphic to some Lie group by means of a smooth local isomorphism. We resume in the following theorem the solution of the Hilbert fifth problem for local topological groups provided by Goldbring. We refer to its work \cite{Go} for the proof and more details.

\begin{theorem}[\cite{Go}] For any locally compact local topological group $\fG$, the conditions listed below are equivalent. \begin{itemize}
\item[a)] $\fG$ is locally Euclidean.
\item[b)] $\fG$ is NSS.
\item[c)] $\fG$ is locally isomorphic to a Lie group.
\end{itemize} \label{H5local} \end{theorem}

\subsubsection{Local (topological and Lie) groups of transformations} \hfill \par

Let $(M,p)$ be a pointed smooth manifold. A {\it local topological group of $\cC^k$-transformations on $(M,p)$} is a tuple $G=(\fG,\eU_{\fG},\W_p,\eW,\Q)$ formed by: \begin{itemize}
\item[i)] a (local) topological group $\fG$ and a neighborhood $\eU_{\fG} \subset \fG$ of the unit;
\item[ii)] a neighborhood $\W_p \subset M$ of $p$;
\item[iii)] an open subset $\eW \subset \eU_{\fG} \times \W_p$ which contains both $\eU_{\fG}\times\{p\}$, $\{e\} \times \W_p$ and a continuous application $\Q: \eW \rar \W_p$;
\end{itemize} such that the following hold: \begin{itemize}
\item[$\bcdot$] for any $(a,b)\in (\eU_{\fG}\times\eU_{\fG}) \cap \eD(\fG)$ and $x \in \W_p$ it holds $$\Q(a,\Q(b,x))=\Q(a\cdot b, x) \,\, ,$$ provided that $(b,x), (a\cdot b,x) \in \eW$ and $(a,\Q(b,x)) \in \eW$;
\item[$\bcdot$] for any $a \in \eU_{\fG}$, the map $\Q(a) \= \Q(a,\cdot)$, defined on the open subset $\eW(a) \= \{x : (a,x) \in \eW\} \subset \W_p$, is of class $\cC^k$;
\item[$\bcdot$] $\Q(e)=\Id_{\W_p}$, i.e. $\Q(e,x)=x$ for any $x \in \W_p$.
\end{itemize}

It follows from the definition that for any $a \in \eU_{\fG}$ there exist a neighborhood $U \subset \W_p$ of $p$ and a neighborhood $V \subset \W_p$ of $\Q(a,p)$ such that $\Q(a)|_U: U \rar V$ is a $\cC^k$-diffeomorphism with inverse given by $(\Q(a)|_U\big)^{-1}=\Q(a^{-1})|_V$. We say that $G$ is {\it almost-effective} (resp. {\it effective}) if the set $$\big\{a \in \eU_{\fG} : \text{$\Q(a)$ fixes a neighborhood of $p$}\big\}$$ is discrete (resp. equal to \{e\}). We say also that $G=(\fG,\eU_{\fG},\W_p,\eW,\Q)$ is {\it locally compact} if $\fG$ is locally compact. We will tacitly assume that $\W_p$, $\eW$ are connected and that $\eW(a)$ is connected for any $a\in \eU_{\fG}$. \smallskip

Two local topological groups $G_i=(\fG_i,\eU_{\fG_i},\W_{p_i},\eW_i,\Q_i)$ of $\cC^k$-transformations acting on $(M_i,p_i)$, with $i=1,2$, are said to be {\it locally $\cC^k$-equivalent} if there exist \begin{itemize}
\item[i)] a neighborhood $\eU_{\zero} \subset \eU_{\fG_1}$ of the unit and a local isomorphism $\f: \fG_1 \rar \fG_2$ defined on $\eU_{\zero}$ with $\f(\eU_{\zero}) \subset \eU_{\fG_2}$;
\item[ii)] two nested neighborhoods $U_{\zero} \subset U \subset \W_{p_1}$ of $p_1$ and an open $\cC^k$-embedding $f: U \rar \W_{p_2}$ with $f(p_1)=p_2$;
\end{itemize} such that the following hold: \begin{itemize}
\item[$\bcdot$] $\eU_{\zero} \times U_{\zero} \subset \eW_1$, $\Q_1(\eU_{\zero} \times U_{\zero}) \subset U$ and $\f(\eU_{\zero}) \times f(U_{\zero}) \subset \eW_2$;
\item[$\bcdot$] for any $(a,x) \in \eU_{\zero}\times U_{\zero}$, it holds that $f\big(\Q_1(a,x)\big)=\Q_2\big(\f(a),f(x)\big)$.
\end{itemize} \smallskip

A local topological group of $\cC^k$-transformations $G=(\fG,\eU_{\fG},\W_p,\eW,\Q)$ on $(M,p)$ is called {\it local Lie group of $\cC^k$-transformations} if $\fG$ is a Lie group and the map $\Q$ is of class $\cC^k$.

\begin{rem} In perfect analogy with what occurs for global groups of transformations, by \cite[Thm 4]{BM} also here the second condition is redundant. Namely, if the (local) topological group $\fG$ is a Lie group and each map $\Q(a):\eW(a) \rar \W_p$ is of class $\cC^k$, then the map $\Q:\eW \rar \W_p$ is automatically of class $\cC^k$. \label{BMT4'} \end{rem}

If $(M,p)$ is equipped with a $\cC^k$-Riemannian metric $g$, then $G=(\fG,\eU_{\fG},\W_p,\eW,\Q)$ is called {\it local topological (resp. Lie) group of isometries} if each map $\Q(a): \eW(a) \rar \W_p$ is a local isometry of $(M,g)$. \medskip

\section{The Myers-Steenrod Theorem in low regularity} \label{dimMS2} \setcounter{equation} 0

As we pointed out in the Introduction, we now provide a proof of the version (MS2') of the Myers-Steenrod Theorem. We also show how it yields to a useful regularity property for homogeneous Riemannian manifolds. First, we recall the following crucial result, which is a consequence of Theorem \ref{H5}.

\begin{theorem}[\cite{MZ2} Thm 2, p. 208] Let $G=(\fG,\Q)$ be a topological group of $\cC^k$-transformations on a smooth manifold $M$, with $k \geq 1$. If $G$ is effective and locally compact, then $G$ is a Lie group of $\cC^k$-transformations. \label{main} \end{theorem}

We also need the following property, which is essentially due to van Dantzig and van der Waerden \cite{DW}. Let $(M,g)$ be a $\cC^{k,\a}$-Riemannian manifold and $\Iso(M,g)$ its full isometry group. We recall that the compact-open topology $\tau_{co}$ on $\Iso(M,g)$ is generated by the basis formed by the sets $$(f;K;\e) \=\{h \in \Iso(M,g): \td_g(f(x),h(x))<\e \text{ for any $x \in K$} \} \,\, ,$$ with $f \in \Iso(M,g)$, $K\subset M$ compact, $\e>0$. On the other hand, the point-open topology $\tau_{po}$ on $\Iso(M,g)$ is generated by the subbasis formed by the sets $$(f;x;\e) \=\{h \in \Iso(M,g): \td_g(f(x),h(x))<\e \} \,\, ,$$ with $f \in \Iso(M,g)$, $x \in M$, $\e>0$.

\begin{lemma} On $\Iso(M,g)$, the compact-open topology coincides with the point-open topology. This topology is Hausdorff, it makes the group operations continuous and it is the coarsest topology with respect to which the action of $\Iso(M,g)$ on $M$ is continuous. Furthermore, with respect to such topology, $\Iso(M,g)$ is locally compact and its action on $M$ is proper. \label{Isoloccpt} \end{lemma}

\begin{proof} We set $G\=\Iso(M,g)$ for short. Let us fix $f \in G$, $K\subset M$ compact, $\e>0$ and let $x_1,{\dots},x_N \in K$ be such that $K \subset \eB_g(x_1,{\textstyle\frac{\e}3})\cup{\dots}\cup \eB_g(x_N,{\textstyle\frac{\e}3})$. We have to show that $A \= \bigcap_{1\leq i\leq N}(f;x_i;{\textstyle\frac{\e}3})$ is contained in $(f;K;\e)$. So, let us consider $h \in A$ and $x \in K$. By construction, there exists $1\leq i \leq N$ such that $\td_g(x,x_i)<{\textstyle\frac{\e}3}$. But then $$\td_g(f(x),h(x)) \leq \td_g(f(x),f(x_i))+\td_g(f(x_i),h(x_i))+\td_g(h(x_i),h(x)) < \e$$ and hence $\tau_{co} \subset \tau_{po}$. Since the other inclusion is obvious, we conclude that $\tau_{co} =\tau_{po}$. The second claim is just a collection of some well known properties of the compact-open topology. We refer to the main theorem of \cite{MaSt} for the last claim. \end{proof}

We are now ready to prove the following

\begin{corollary}[Enhanced version of the Myers-Steenrod Theorem] Any closed group of isometries of a $\cC^{k,\a}$-Riemannian manifold, with $k+\a>0$, is a Lie group of isometries. \label{MS} \end{corollary}
\begin{proof}Let $(M,g)$ be a $\cC^{k,\a}$-Riemannian manifold, with $k+\a>0$, and consider its full isometry group $G=\Iso(M,g)$. Then, by means of Theorem \ref{Isoreg} and Lemma \ref{Isoloccpt}, $G$ is an effective group of $\cC^{k+1}$-transformation and $G$ is locally compact. Then, by Theorem \ref{main}, it is a Lie group of isometries and the thesis follows. \end{proof}

This corollary yields to the following improvement of a well known property of homogeneous Riemannian manifolds. As usual, a $\cC^{k,\a}$-Riemannian manifold $(M,g)$ is called {\it homogeneous} if it admits a closed, transitive group of isometries.

\begin{theorem} Any homogeneous $\cC^{0,\a}$-Riemannian manifold, with $\a>0$, is real analytic. \label{reghom} \end{theorem}

\begin{proof} Let $(M,g)$ be a $\cC^{0,\a}$-Riemannian manifold, with $\a>0$, and $G=(\fG,\Q)$ a closed, transitive topological group of isometries acting on $(M,g)$. Pick a distinguished point $x_{\zero} \in M$ and consider the {\it isotropy subgroup of $G$ at $x_{\zero}$}, i.e. $\fH \= \{a \in \fG: \Q(a,x_{\zero})=x_{\zero}\}$. From Corollary \ref{MS}, it follows that $G$ is a Lie group of isometries and, by means of Theorem \ref{Isoreg}, the map $\Q: \fG \times M \rar M$ is of class $\cC^1$. Then, $\fH$ is an embedded Lie subgroup of $\fG$ and we get the $\cC^1$-diffeomorphism \beq \q_{x_{\zero}}: \fG/\fH \rar M \,\, , \quad \q_{x_{\zero}}(a\fH)\=\Q(a,x_{\zero}) \,\, . \eeq Since $G$ acts by isometries, there exists a unique invariant $\cC^{\w}$-Riemannian metric $\tilde{g}$ on $\fG/\fH$ which makes the map $\q_{x_{\zero}}: (\fG/\fH,\tilde{g}) \rar (M,g)$ an isometry. From this the thesis follows. 
\end{proof} \medskip

\section{Proof of Theorem \ref{MAIN1}} \label{thmA} \setcounter{equation} 0

The purpose of this section is to give the proof of a local analogue of Theorem \ref{main}, namely

\begin{theorem} Let $G=(\fG,\eU_{\fG},\W_p,\eW,\Q)$ be a local topological group of $\cC^k$-transformations on a pointed smooth manifold $(M,p)$, with $k\geq1$. If $G$ is locally compact and effective, then $G$ is a local Lie group of $\cC^k$-transformations. \label{mainloc} \end{theorem}

\noindent of which Theorem \ref{MAIN1} is an immediate consequence. \smallskip

First, we need a preparatory lemma. For its statement, we introduce the following definition. Let $\fG$ be a local topological group. For any integer $N \geq 1$ and for any $a_1,{\dots},a_N,b \in \fG$, we say that {\it the element $a_1\cdot a_2 \cdot {\dots} \cdot a_N$ is well defined and equal to $b$}, for short $a_1\cdot a_2 \cdot {\dots} \cdot a_N = b$, if the following condition defined by induction on $N$ is satisfied: for any $1\leq i \leq N$ there exist $b_i, b_i' \in \fG$ such that $a_1\cdot {\dots} \cdot a_i = b_i$, $a_{i+1} \cdot {\dots} \cdot a_N = b_i'$, $(b_i', b_i'') \in \eD(\fG)$ and $b_i\cdot b_i'=b$. If $\eU \subset \fG$ is a neighborhood of the unit such that $a_1\cdot {\dots} \cdot a_N$ is well defined for any choice of $a_1, {\dots}, a_N \in \eU$, we set $\eU^N \=\big\{a_1 \cdot {\dots} \cdot a_N : a_1, {\dots}, a_N \in \eU \big\}$.

\begin{lemma} Let $G=(\fG,\eU_{\fG},\W_p,\eW,\Q)$ be a local topological group of $\cC^k$-transformations on a pointed smooth manifold $(M,p)$. Then: \begin{itemize}
\item[i)] For any compact set $K \subset \W_p$, there exists a neighborhood $\eU \subset \eU_{\fG}$ of the unit such that $\eU \times K \subset \eW$.
\item[ii)] For any fixed $N \in \bN$, there exists a neighborhood $\eW_N$ of $\{e\} \times \W_p$ in $\eW$ such that for any $$(a_1,x), \dots, (a_N,x) \in \eW_N$$ the element $a_1\cdot {\dots} \cdot a_N$ is well defined and $(b_i\cdot b_i',x), (b_i,\Q(b_i',x)) \in \eW$ for any $1 \leq i \leq N$, where $b_i \= a_1\cdot {\dots} \cdot a_i$ and $b_i' \= a_{i+1} \cdot {\dots} \cdot a_N$.
\end{itemize} \label{star} \end{lemma}
\begin{proof} To prove the first claim, it is sufficient to observe that, since $\{e\}\times K$ is compact, there exists an finite open cover $\{\cI_1,{\dots},\cI_{\ell}\}$ of $\{e\}\times K$ inside $\eW$, where the open sets $\cI_i$ have the form $\cI_i=\eU_i \times U_i$ for any $1\leq i\leq {\ell}$. Then $\eU \subset \bigcap_{1\leq i\leq {\ell}} \eU_i$ satisfies (i).

We now recall that there exists a sequence of nested neighborhoods $$\{e\} \subset {\dots} \subset \wt{\eD}_N(\fG) \subset \wt{\eD}_{N-1}(\fG) \subset {\dots} \subset \wt{\eD}_2(\fG) \subset \fG$$ of the unit such that, for any $N\geq2$ and for any choice of $N$ elements $a_1,{\dots},a_N \in \wt{\eD}_N(\fG)$, the product $a_1\cdot {\dots} \cdot a_N$ is well defined (see e.g. \cite[Lemma 2.5]{Go}).

Fix $N \in \bN$. By (i) we can consider $N$ exhaustions $\big\{U^{(n)}_1\big\}, \dots, \big\{U^{(n)}_N\big\}$ of $\W_p$ by relatively compact open sets and two sequences $\big\{\eU^{(n)}\big\}$, $\big\{\eU'{}^{(n)}\big\}$ of neighborhoods of the unit in $\eU_{\fG} \subset \fG$ such that: \begin{itemize}
\item[$\bcdot$] $U^{(n)}_1 \subset\subset U^{(n)}_2 \subset\subset \dots \subset\subset U^{(n)}_N \subset\subset U^{(n+1)}_1$ and $\eU^{(n+1)}\subset \eU'{}^{(n+1)} \subset\eU^{(n)}$,
\item[$\bcdot$] $\eU'{}^{(n)} \times U^{(n)}_N \subset \big(\wt{\eD}_N(\fG) \times M\big) \cap \eW$,
\item[$\bcdot$] $\big(\eU^{(n)}\big)^N \subset \eU'{}^{(n)}$ and $\Q\big(\eU^{(n)}\times U^{(n)}_i\big)\subset U^{(n)}_{i+1}$ for any $1 \leq i \leq N-1$.
\end{itemize} It is immediate now to realize that for any $(a_1,x), \dots, (a_N,x) \in \eU^{(n)}\times U^{(n)}_1$ it holds that $a_1\cdot {\dots} \cdot a_N$ is well defined and that $(b_i\cdot b_i',x), (b_i,\Q(b_i',x)) \in \eW$ for any $1 \leq i \leq N$, with $a_1\cdot {\dots} \cdot a_i = b_i$ and $a_{i+1} \cdot {\dots} \cdot a_N = b_i'$. Therefore, if we define $\eW_N \= \bigcup_{k \in \bN} \eU^{(n)}\times U^{(n)}_1$, then claim (ii) follows. \end{proof}

We observe that Theorem \ref{mainloc} involves only local object. Hence, without loss of generality, we may assume that $(M,p)= (\bR^m,0)$. However, for the sake of clarity, in what follows we will still use the symbols $p$ and $\W_p$ for $0$ and the distinguished neighborhood of $0$, respectively.

\begin{prop} Let $k\geq1$ and $G=(\fG,\eU_{\fG},\W_p,\eW,\Q)$ a locally compact local group of $\cC^k$-transformations on $(M,p)= (\bR^m,0)$. Then, there exist a relatively compact neighborhood $\eV \subset \eU_{\fG}$ of the unit and a ball $B \subset \W_p$ centered at $p$ which satisfy the following property: if $\fH$ is a subgroup of $\fG$ entirely contained in $\eV$, then there exists a neighborhood $V_{\zero} \subset B$ of the origin such that $\Q(a)|_{V_{\zero}}=\Id_{V_{\zero}}$ for any $a \in \fH$. \label{prop2} \end{prop}

\begin{proof} By \cite[Thm 1]{Mon} and \cite[p. 685]{BM}, given $(a,x) \in \eW$, for any neighborhood $\eV_a \subset \eU_{\fG}$ of $a$ and for any ball $B \subset \W_p$ centered at $x$ such that $\eV_a \times B \subset \eW$, the following holds: {\it every partial derivative of the function $\Q(b)|_B: B \rar \bR^m$ up to order $k$ is continuous with respect to $b \in \eV_a$}. 

Since $\Q(e)$ is the identity map of $\W_p \subset \bR^m$, from Lemma \ref{star} it follows that there exist a relatively compact neighborhood $\eV \subset \fG$ of the unit and a ball $B \subset\subset \W_p$ of the origin such that $\bar{\eV} \times \bar{B} \subset \eW_N$, with $N \geq 2$ large enough, and the family of functions $\big\{(\Q(a)-\Id)|_{B}: B \rar \bR^m \big\}_{a \in \eV}$ is uniformly bounded in the Banach space $\cC^k(\bar{B})$ by a positive constant $C \in \bR$, which can be taken as small as one likes by restricting $\eV$. Let now $\fH$ be a subgroup of $\fG$ entirely contained in $\eV$. By taking the closure, one can suppose that $\fH$ is closed and hence compact. We define the map $$T: B \rar \bR^m \,\, , \quad T(x) \= \int_{\fH} \Q(a,x)\,d\l(a) \,\, ,$$ where $\l$ is the Haar measure of $\fH$, normalized in such a way that $\l(\fH)=1$. By differentiating under the integral sign, it follows that $T$ is of class $\cC^k$. Moreover $$|\!|T-\Id|\!|_{\cC^k(\bar{B})} \leq \int_{\fH} |\!|\Q(a)-\Id|\!|_{\cC^k(\bar{B})}\,d\l(a) \leq C \,\, .$$ By the Inverse Function Theorem, there exists an open neighborhood $V \subset B$ of the origin such that the restriction $T|_V:V \rar \bR^m$ is an open $\cC^k$-embedding and $T(V)\subset B$. On the other hand, we can choose a sufficiently small neighborhood $V_{\zero} \subset V$ of the origin such that $\Q(a)(V_{\zero}) \subset V$ for any $a \in \eV$. Then, from the bi-invariance of the Haar measure, for any $b \in \fH$ and for any $x \in V_{\zero}$ it follows that \begin{align*}\big(T \circ \Q(b)\big)(x) &= \int_{\fH} \big(\Q(a)\circ\Q(b)\big)(x)\,d\l(a) \\ &= \int_{\fH} \Q(a\cdot b)(x)\,d\l(a) = \int_{\fH} \Q(a)(x)\,d\l(a) = T(x) \,\, . \end{align*} Since $T$ is invertible in $V$, we get $\Q(b)|_{V_{\zero}}=\Id_{V_{\zero}}$ for any $b \in \fH$. \end{proof}

We are now able to conclude the proof of Theorem \ref{mainloc}. Suppose that $G$ is a locally compact and effective local topological group of $\cC^k$-transformations on $(\bR^m,0)$. From Proposition \ref{prop2}, we directly get that the abstract (local) group of $G$ is NSS. By Theorem \ref{H5local} and Remark \ref{BMT4'}, we get the thesis. \medskip

\section{Proof of Theorem \ref{MAIN2}} \label{thmB} \setcounter{equation} 0

We firstly recall that a $\cC^{k,\a}$-Riemannian manifold $(M,g)$ is said to be {\it locally homogeneous} if the pseudogroup of local isometries of $(M,g)$ acts transitively on $M$, i.e. if for any $x,y \in M$ there exist two open sets $U_x,U_y \subset M$ and a local isometry $f: U_x \rar U_y$ such that $x \in U_x$, $y \in U_y$ and $f(x)=y$.

Secondly, consider a local topological group of transformations $G=(\fG,\eU_{\fG},\W_p,\eW,\Q)$ on a pointed smooth manifold $(M,p)$. We recall that the {\it orbit of $G$ through $p$} is the set \begin{multline*} G(p) \= \Big\{\big(\Q(a_1) \circ{\dots}\circ\Q(a_{N})\big)(p) : N\geq1 \, , \,\, a_i \in \eU_{\fG} \text{ for any $1\leq i \leq N$} \, , \\ \big(\Q(a_{j+1}) \circ{\dots}\circ\Q(a_{N})\big)(p) \in \eW(a_j) \text{ for any $1\leq j \leq N-1$} \Big\} \,\, .\end{multline*} Motivated by the terminology for Lie algebra actions, we say that $G$ is {\it transitive} if $G(p)$ contains a neighborhood of the point $p$.

The above properties of local homogeneity for Riemannian manifolds and of transitivity for local groups of isometries are related as follows. If $(M,g)$ is a smooth locally homogeneous Riemannian manifold, then it is real analytic (see e.g. \cite[Thm 2.2]{Sp1} or \cite[Lemma 2.1]{BLS}). From this and \cite{No}, it follows that for any $p \in M$ there exists a local Lie group of isometries which acts transitively on $(M,g,p)$. Notice that our Theorem \ref{MAIN2} is a kind of converse of such a claim. \smallskip

Since we deal with locally homogeneous manifolds, in order to prove Theorem \ref{MAIN2} we need to define rigorously a {\it local analogous} of the usual quotient of Lie groups. In this direction, the following proposition details the construction sketched in \cite[Sec 3.1]{Ped1}.

\begin{prop} Let $\fG$ be a Lie group and $\fH \subset \fG$ be a (not necessarily closed) Lie subgroup. \begin{itemize}
\item[a)] There exist a neighborhood $\eU_{\fH} \subset \fH$ of the unit in the manifold topology of $\fH$ and two neighborhoods $\eU, \eV \subset \fG$ of the identity such that: $\eU_{\fH}$ is a sub-local group of $\fG$ with associated neighborhood $\eV$, $\eU_{\fH}$ is closed in $\eV$ and $\eU^6 \subset \eV$.
\item[b)] The binary relation on $\eU$ defined by $$a \sim b \overset{\rm def}{\iff} a^{-1}\cdot b \in \eU_{\fH}$$ is an equivalence relation on $\eU$ and the equivalence class $[a]_{\sim}$ of $a \in \eU$ verifies $[a]_{\sim} = (a\eU_{\fH}) \cap \eU$.
\item[c)] The quotient space $(\fG/\fH)_{(\eU_{\fH},\eU,\eV)} \= \eU/{\sim} = \big\{(a\eU_{\fH}) \cap \eU : a \in \eU\big\}$ is a topological manifold and it admits a real analytic structure, which is unique up to $\cC^{\w}$-diffeomorphism, with respect to which the following conditions hold: \begin{itemize}
\item[$\bcdot$] the canonical projection $\pi_{(\eU_{\fH},\eU,\eV)}: \eU \rar (\fG/\fH)_{(\eU_{\fH},\eU,\eV)}$ is a $\cC^{\w}$-submersion;
\item[$\bcdot$] the tuple $G_{(\fG,\fH),(\eU_{\fH},\eU,\eV)}\=(\fG,\eU,(\fG/\fH)_{(\eU_{\fH},\eU,\eV)},\eW,\Q)$ with $$\begin{gathered} \eW \= \big\{\big(a,(b\eU_{\fH}) \cap \eU\big) : a \in \eU ,\, b \in \eU ,\, a\cdot b \in \eU \big\} \,\, , \\ \Q: \eW \rar (\fG/\fH)_{\eU_{\fH},\eU,\eV} \,\, , \quad \Q(a)\big((b\eU_{\fH}) \cap \eU\big) \= ((a \cdot b)\eU_{\fH}) \cap \eU \end{gathered}$$ is a local Lie group of $\cC^{\w}$-transformations acting transitively on $\big((\fG/\fH)_{(\eU_{\fH},\eU,\eV)},(e\eU_{\fH})\cap\eU\big)$.
\end{itemize} 
\item[d)] If $(\eU_{\fH}, \eU, \eV)$ and $(\eU'_{\fH}, \eU', \eV')$ are two triples both satisfying all conditions in (a), then $G_{(\fG,\fH),(\eU_{\fH},\eU,\eV)}$ is locally $\cC^{\w}$-equivalent to $G_{(\fG,\fH),(\eU'_{\fH},\eU',\eV')}$.
\end{itemize} \label{locfact} \end{prop}
\begin{proof} The proof of (a) is straightforward, while (b) is the statement of \cite[Lemma 2.13]{Go}. To prove (c), one can easily adapt the well known proof of the corresponding statement for the quotient of a Lie group with respect to a closed subgroup (see e.g. \cite[Ch II, Sec 4]{Hel}). Finally, to prove (d), let us consider two neighborhoods $\eU_1,\eU_2 \subset \fG$ of the unit such that $(\eU_1)^2 \subset \eU_2$ and $\eU_{\fH} \cap \eU_2= \eU'_{\fH} \cap \eU_2$. Then let us pick a neighborhood $\eU_{\zero} \subset \eU \cap \eU' \cap \eU_1$ of the unit in $\fG$. One can directly check that the map $$\pi_{(\eU_{\fH},\eU,\eV)}(\eU_{\zero}) \rar \pi_{(\eU'_{\fH},\eU',\eV')}(\eU_{\zero}) \,\, , \quad (a\eU_{\fH}) \cap \eU \mapsto (a\eU'_{\fH}) \cap \eU'$$ is a $\cC^{\w}$-diffeomorphism. \end{proof}

Given a Lie group $\fG$ together with a Lie subgroup $\fH$, we call {\it admissible triple for $\fH$ in $\fG$} any choice of $(\eU_{\fH}, \eU, \eV)$ as in (a) and {\it local factor space of $\fG$ modulo $\fH$} any quotient $(\fG/\fH)_{(\eU_{\fH}, \eU, \eV)}$ as in (c). Notice that $\fH$ is closed in $\fG$ if and only if $(\fH,\fG,\fG)$ is an admissible triple for $\fH$ in $\fG$ and, in that case, $(\fG/\fH)_{(\fH,\fG,\fG)}=\fG/\fH$. For other details concerning local factor spaces and locally homogeneous metrics, see \cite{Mos} and \cite{Sp2}. \smallskip

Let now $G=(\fG,\eU_{\fG},\W_p,\eW,\Q)$ be an almost-effective local Lie group of $\cC^k$-transformations on $(M,p)$ and suppose that $k\geq2$. Let also $\gg \= \Lie(\fG)$ and $\exp:\gg \rar \fG$ the Lie exponential of $\fG$. For any $X \in \gg$, we consider the open set $$\eW^X \= \{(t,x) \in \bR \times \W_p : (\exp(tX),x) \in \eW \}$$ and the map of class $\cC^k$ $$\Q^X : \eW^X \rar M \,\, , \quad \Q^X(t,x) \= \Q\big(\exp(tX),x\big) \,\, .$$ This allows to consider the differential \beq \Q_*: \gg \rar \cC^{k-1}(\W_p;TM|_{\W_p}) \,\, , \quad \Q_*(X)_x \= \frac{d}{dt}\Q^X(t,x)\Big|_{t=0} \,\, . \eeq In full analogy with the theory of Lie group actions, one can prove that the map $\Q_*$ is $\bR$-linear, injective and, for any $X,Y \in \gg$ \begin{gather}
\Q_*\big(\!\Ad(a)X\big)_{\Q(a,x)}=\diff(\Q(a))_x\big(\Q_*(X)\big)_x \quad \text{ for any $(a,x) \in \eW$} \,\, , \nonumber \\
\Q_*\big([X,Y]\big)=-[\Q_*(X),\Q_*(Y)] \,\, . \label{brack}
\end{gather} Let us now define $\gh \= \big\{X \in \gg : \Q_*(X)_p=0 \big\}$. By \eqref{brack}, it follows that $\gh$ is a Lie subalgebra of $\gg$ and so we can consider the unique connected Lie subgroup $\fH$ of $\fG$ such that $\Lie(\fH)=\gh$. We call it the {\it abstract isotropy subgroup of $G$ at $p$}. Notice that $G$ is almost-effective if and only if $\gh$ does not contain any non-trivial ideal of $\gg$, while $G$ is effective if and only if $\fH$ does not contain any non-trivial normal subgroup of $\fG$. As expected, the following proposition holds.

\begin{prop} For any admissible triple $(\eU_{\fH},\eU,\eV)$ for $\fH$ in $\fG$, $G$ is locally $\cC^k$-equivalent to the local Lie group of $\cC^{\w}$-transformations $G_{(\fG,\fH),(\eU_{\fH},\eU,\eV)}$ defined in (c) of Proposition \ref{locfact}. \label{loceqG} \end{prop}

\begin{proof} Let $(\eU_{\fH},\eU,\eV)$ be an admissible triple for $\fH$ in $\fG$ and choose a sufficiently small neighborhood of the unit $\eU_{\zero} \subset \eU \cap \eU_{\fG}$. Then, the identity map $\Id_{\fG}: \fG \rar \fG$ and the application $$\q_{x_{\zero}}:\pi_{(\eU_{\fH},\eU,\eV)}(\eU_{\zero}) \rar M \,\, , \quad \q_{x_{\zero}}((a\eU_{\fH})\cap \eU) \= \Q(a,p)$$ give rise to a local $\cC^k$-equivalence between $G$ and $G_{(\fG,\fH),(\eU_{\fH},\eU,\eV)}$. \end{proof}

Let now $(M,g)$ be a locally homogeneous $\cC^1$-Riemannian manifold and assume that there exist a point $p\in M$ and a locally compact, effective local topological group of isometries $G=(\fG,\eU_{\fG},\W_p,\eW,\Q)$ which acts transitively on $(M,g,p)$.

\begin{lemma} For any fixed $x_{\zero} \in M$, there exists a neighborhood $U_{x_{\zero}} \subset M$ of $x_{\zero}$ and an open $\cC^2$-embedding $\f_{x_{\zero}}: U_{x_{\zero}} \rar \bR^m$ such that the pulled-back metric $(\f_{x_{\zero}}^{-1})^*g$ on the open set $\f_{x_{\zero}}(U_{x_{\zero}}) \subset \bR^m$ is real analytic. \label{AC2lh} \end{lemma}

\begin{proof} Since $(M,g)$ is locally homogeneous, it is sufficient to prove the claim for $x_{\zero}=p$. By means of Theorem \ref{MAIN1} and Theorem \ref{Isoreg}, $G$ is a local Lie group of isometries and the map $\Q$ is of class $\cC^2$. Then, let $\fH$ be the abstract isotropy of $G$ at $p$ and pick an admissible triple $(\eU_{\fH}, \eU, \eV)$ for $\fH$ in $\fG$. By means of Proposition \ref{loceqG}, $G$ is locally $\cC^2$-equivalent to the local Lie group $G_{(\fG,\fH),(\eU_{\fH},\eU,\eV)}$ of $\cC^{\w}$-transformations on the local factor space $(\fG/\fH)_{(\eU_{\fH},\eU,\eV)}$. Since $G$ acts on $(M,g,p)$ by isometries, there exists a unique $\cC^{\w}$-Riemannian metric $\tilde{g}$ on $(\fG/\fH)_{(\eU_{\fH},\eU,\eV)}$ which makes the map $\q_{x_{\zero}}:\pi_{(\eU_{\fH},\eU,\eV)}(\eU_{\zero}) \rar M$ defined in the proof of Proposition \ref{loceqG} a local isometry. \end{proof}

We may now conclude the proof of Theorem \ref{MAIN2}. By Lemma \ref{AC2lh}, there exists a $\cC^2$-atlas $\eA=\{(U_\a,\xi_\a)\}$ on $M$ such that the metric $g$ is real analytic with respect to each coordinate chart $(U_\a,\xi_\a) \in \eA$. But then by \cite[Lemma 1.2]{DK} there exists a $\cC^{\w}$-atlas $\eA'$ on $M$ which is compatible with $\eA$ and with respect to which $g$ is real analytic.

\bigskip\bigskip\bigskip
\font\smallsmc = cmcsc8
\font\smalltt = cmtt8
\font\smallit = cmti8
\hbox{\parindent=0pt\parskip=0pt
\vbox{\baselineskip 9.5 pt \hsize=5truein
\obeylines
{\smallsmc
Dipartimento di Matematica e Informatica ``Ulisse Dini'', Universit$\scalefont{0.55}{\text{\Aac}}$ di Firenze
Viale Morgagni 67/A, 50134 Firenze, ITALY}
\smallskip
{\smallit E-mail adress}\/: {\smalltt francesco.pediconi@unifi.it
}
}
}

\end{document}